\documentclass[11pt,a4paper,twoside]{article}
\usepackage{amsmath, amsthm, amscd, amsfonts, amssymb, graphicx, color}
\usepackage[bookmarksnumbered, colorlinks, plainpages]{hyperref}
\input{mathrsfs.sty}

\newtheorem{theorem}{Theorem}[section]
\newtheorem*{thma}{Theorem A}
\newtheorem*{thmb}{Theorem B}
\newtheorem*{thmc}{Theorem C}
\newtheorem*{thmd}{Theorem D}

\newtheorem{lemma}[theorem]{Lemma}

\newtheorem{corollary}[theorem]{Corollary}
\theoremstyle{definition}

\theoremstyle{remark}

\newtheorem{remark}[theorem]{Remark}
\begin{document}
\bigskip \title{\Large {\bf Convolution properties of univalent harmonic mappings convex in one direction}}
\author{{Raj Kumar\,$^a$\thanks{e-mail: rajgarg2012@yahoo.co.in}\,, Sushma Gupta\,$^b$ ~and~ Sukhjit Singh\,$^b$ }}
\date{}
\maketitle

\begin{abstract}  Let $\ast$ and $\widetilde{\ast}$ denote the convolution of two analytic  maps and that of an analytic map and a harmonic map respectively. Pokhrel [\ref{po}] proved that if $f=h+\overline{g}$ is a harmonic map convex in the direction of $e^{i\gamma}$ and $\phi$ is an analytic map in the class  $DCP,$ then $f\widetilde{\ast}\phi=h\ast \phi+\overline{g\ast \phi}$  is also convex in the direction of $e^{i\gamma},$ provided $f\widetilde{\ast} \phi$ is locally univalent and sense-preserving. In the present paper we obtain a general condition under which $f\widetilde{\ast}\phi$ is locally univalent and sense-preserving. Some interesting applications of the general result are also presented.
\end{abstract}
\section{Introduction}
 Let $\mathcal{H}$ denote the class of all complex valued harmonic mappings $f=h+\overline{g}$ defined in the unit disk $E = \{z: |z|<1\}$. Such harmonic mappings are locally univalent and sense-preserving if and only if $h'\not=0$ in $E$ and the dilatation function $\omega$, defined by $\displaystyle\omega={g'}/{h'}$, satisfies $|\omega(z)|<1$ for all $z\,\in E$. The class of all univalent harmonic and sense-preserving mappings $f=h+\overline g$ in $E,$ normalized by the conditions $ f(0)=0 $ and $f_{z}(0)=1,$ is denoted by $S_H$. Therefore, a function $f=h+\overline g$ in the class $S_H$ has the representation
\begin{equation} f(z) = z+ \sum _{n=2}^{\infty} a_nz^n + \sum _{n=1}^{\infty}\overline{ b}{_n}\overline{z}^n,\,\,\,z\,\in\, E.\end{equation}  If the co-analytic part $g(z)\equiv0$ in $E$, then the class $S_H$ reduces to the usual class $S$ of all normalized univalent analytic functions. We denote by $K_H $ and $C_H$ the  subclasses of $S_H$ consisting of those functions which map $E$ onto convex and close-to-convex domains, respectively. $K$ and $C$ will denote corresponding subclasses of $S$. A domain $\Omega$ is said to be convex in a direction $e^{i\gamma},\, 0 \leq \gamma < \pi,$ if every line parallel to the line through $0$ and $ e ^{i \gamma}$ has either connected or empty intersection with $\Omega$. Let $K_{H(\gamma)}$ and $K_{\gamma},$ $0 \leq \gamma < \pi$ denote the subclass of $S_H$ and $S$ respectively, consisting of functions which map the unit disk $E$ on to domains convex in the direction of $e^{i\gamma}$.  In particular a domain convex in the horizontal direction will be denoted by $CHD$. Hengartner and Schober [\ref{he and sc}] characterized the mappings in $K_{\pi/2}$ as under:
\begin{thma} Suppose $f$ is analytic and non constant in $E$. Then $$\Re[(1-z^2)f'(z)]\geq0,\,\,z\in E$$ if and only if\\(i) $f$ is univalent in $E$;\\(ii) $f$ is convex in the direction of the imaginary axis;\\ (iii) there exist sequences $\{z_n'\}$ and $\{z_n''\}$ converging to $z=1$ and $z=-1$, respectively, such that\\ $lim_{n\rightarrow \infty} \,\,\Re(f(z_n'))= sup_{|z|<1}\,\,\Re(f(z))$ and  $ lim_{n\rightarrow \infty} \,\,\Re(f(z_n''))= inf_{|z|<1}\,\,\Re(f(z)). $
\end{thma}
\indent The convolution or the Hadamard product $\phi {\ast} \psi $ of two analytic mappings $ \phi(z)=\sum_{n=0}^\infty a_n z^n $ and $\psi(z) = \sum_{n=0}^\infty A_n z^n$ in $E$ is defined as $(\phi {\ast} \psi)(z)=\sum_{n=0}^\infty a_n A_n z^n,\,z\in E.$ The convolution of a harmonic function $f=h+\overline{g}$ with an analytic function $\phi$, both defined in $E,$ is denoted by $f\widetilde{\ast}\phi$ and is defined as: $$f\widetilde{\ast}\phi=h\ast \phi+\overline{g\ast \phi}.$$ Clunie and Sheil-Small  [\ref{cl and sh}]  proved that if $\phi\in K$ and $F\in K_H,$ then for every $\alpha,\,|\alpha|\leq1$ $(\alpha\overline\phi + \phi)\ast F\in C_H.$ They also posed the question: if $F\in K_H$, then what is the collection of harmonic functions  $f$, such that $ F\ast f\in K_H$? Ruscheweyh and Salinas [\ref{ru and sa}] partially answered the question of Clunie and Sheil-Small by introducing a class $DCP$(Direction convexity preserving). An analytic function $\phi$ defined in $E$ is said to be in the class $DCP$ if for every $f$ in $K_\gamma$, $\phi \ast f\in K_\gamma.$ Functions in the class $DCP$ are necessarily convex in $E$.
 \begin{thmb} ([\ref{ru and sa}]) Let $\phi$ be analytic in $E$.  Then $F \widetilde{\ast} \phi \in K_H $ for all $F\in K_H$ if and only if $\phi$ is in $DCP$.
 \end{thmb}
  Pokhrel [\ref{po}] further investigated convolution properties of functions in the class $DCP$ and proved the following:
\begin{thmc} ([\ref{po}]) Let $f=h+\overline{g}$ be in $K_{H(\gamma)}$, $0 \leq \gamma < \pi$. Then for an analytic function $\phi$ in $DCP,$ $f\widetilde{\ast}\phi\in K_{H(\gamma)},$ provided $f\widetilde{\ast}\phi$ is locally univalent and sense-preserving in $E$.\end{thmc}
 In the present paper we find a condition on a univalent harmonic function convex in one direction so that its convolutions with functions of the class $DCP$ are also convex in the same direction. Some interesting applications are also presented.
\section{Preliminaries}
Following lemmas will be required to prove our main results.
\begin{lemma} Let $\phi$ and $\xi$ be analytic in $E$ with $\phi(0)=\xi(0)=0$ and $\phi'(0)\xi'(0)\not=0$.  Suppose that for each $\beta(|\beta|=1)$ and $\sigma(|\sigma|=1)$ we have $$\phi(z) \ast \frac{1+\beta\sigma z}{1-\sigma z}\xi(z)\not=0,\,\,\, 0<|z|<1.$$ Then for each function $F$ analytic in $E,$ $(\phi\ast F\xi)/(\phi\ast\xi)$ takes values in the convex hull of $F(E)$.\end{lemma}
\begin{lemma} Let $\chi$ be analytic in $E$ with $\chi(0)=0$ and suppose that there exist constants $\zeta$ and $\eta$ with $|\zeta|=|\eta|=1$ such that for each $z$ in $E$ $$ \Re\left[(1-\zeta z)(1-\eta z)\frac{\chi(z)}{z}\right]>0.$$ Then for every convex function $\phi$, $\phi(z)\ast \chi(z)\not=0\,\,(0<|z|<1).$ \end{lemma}
\begin{lemma} A locally univalent harmonic function $f=h+\overline g$  in $E$ is a univalent harmonic mapping of $E$ onto a domain convex in a direction $e^{i\gamma}$ if and only if $h-e^{2i\gamma}g$ is a univalent analytic mapping of $E$ onto a domain convex in the direction $e^{i\gamma}$.\end{lemma}
Lemma 2.1 and Lemma 2.2 are due to Ruscheweyh and Sheil-Small [\ref{ru and sh}] whereas Lemma 2.3 is due to Clunie and Sheil-Small [\ref{cl and sh}].
\section{Main Results}
We begin with the following result.
\begin{theorem} Let $f=h+\overline{g}\in K_{H(\gamma)},$ $0 \leq \gamma < \pi$. If for some complex constants $\eta$ and $\xi$ with  $|\xi|=|\eta|=1,$ \begin{equation}\Re\left[(1-\eta z)(1-\xi z)(h'-e^{2i\gamma}g')\right]>0,\,\, {\rm in}\,\, E,\end{equation}  then $f\widetilde\ast \phi \in K_{H(\gamma)}$ for every analytic function $\phi\in DCP$.\end{theorem}
\begin{proof} In view of Theorem C, we only need to show that $f\widetilde\ast \phi=h\ast \phi+\overline{g\ast\phi}$ is locally univalent and sense-preserving in $E$ i.e., the dilatation $\widetilde \omega=(\phi\ast g)'/(\phi\ast h)'$ of $f\widetilde\ast \phi$ satisfies $|\widetilde \omega(z)|<1$ in $E$. For this it suffices to show that $\Re ((1+e^{2i\gamma}\widetilde \omega)/(1- e^{2i\gamma}\widetilde \omega))>0$ in $E$. Now
$$ \begin{array}{llll}
\vspace{.2cm}
\displaystyle \indent \hspace{-3.5cm} \Re\left[\frac{1+e^{2i\gamma}\widetilde \omega}{1-e^{2i\gamma}\widetilde \omega}\right]=\Re\left[\frac{(\phi\ast h)'+e^{2i\gamma}(\phi\ast g)'}{(\phi\ast h)'-e^{2i\gamma}(\phi\ast g)'}\right]\\
\vspace{.2cm}
\displaystyle \indent \hspace{-1.3cm}= \Re\left[\frac{\phi\ast z(h'+ e^{2i\gamma}g')}{\phi\ast z(h'- e^{2i\gamma}g')}\right]\\
\vspace{.2cm}
\displaystyle \indent \hspace{-1.3cm}= \Re\left[\frac{\phi\ast z(h'- e^{2i\gamma}g')\frac{(h'+ e^{2i\gamma}g')}{(h'- e^{2i\gamma}g')}}{\phi\ast z(h'-e^{2i\gamma} g')}\right]\\
\displaystyle \indent \hspace{-1.3cm}= \Re\left[\frac{\phi\ast z(h'- e^{2i\gamma}g')P}{\phi\ast z(h'-e^{2i\gamma} g')}\right],
\end{array}$$
where $\displaystyle P=(h'+e^{2i\gamma} g')/(h'- e^{2i\gamma}g')=(1+e^{2i\gamma} g'/h')/(1- e^{2i\gamma}g'/h').$ Since $g'/h'$ is the dilatation of $f,$ therefore $|g'/h'|<1$ in $E$ as $f$ is univalent and hence locally univalent in $E$. Thus $\Re (P)>0$ in $E$ and therefore, in view of Lemma 2.1, we shall get the desired result if for each $\beta(|\beta|=1)$ and $\sigma(|\sigma|=1)$ \begin{equation}
\displaystyle \phi(z) \ast \frac{1+\beta\sigma z}{1-\sigma z} z(h'(z)-e^{2i\gamma} g'(z))\not=0,\,0<|z|<1.\end{equation}
 As $\phi$ is in the class $DCP,$ so $\phi$ is convex analytic in $E$. Therefore, to prove (3) we shall apply Lemma 2.2 and show that for some constants $\zeta,\eta$ with $|\zeta|=|\eta|=1$, $$\displaystyle \Re\left[(1-\zeta z)(1-\eta z)\frac{\frac{1+\beta\sigma z}{1-\sigma z} z(h'(z)-e^{2i\gamma}g'(z))}{z}\right]>0,\,z\in E.$$
By setting $\zeta=\sigma$ and $\displaystyle \beta=\frac{-\xi}{\sigma}\,(|\xi|=1)$ we have $$\indent\hspace{-5.1cm}\Re\left[(1-\zeta z)(1-\eta z)\frac{\frac{1+\beta\sigma z}{1-\sigma z} z(h'-e^{2i\gamma}g')}{z}\right]=\Re\left[(1-\eta z)(1-\xi z)(h'-e^{2i\gamma}g')\right]$$ $$\indent\hspace{-.1cm}>0\,\,{\rm in}\,\, E\,\,\,{\rm(in\, view\, of\, (2)}).$$
Hence $\Re ((1+e^{2i\gamma}\widetilde \omega)/(1-e^{2i\gamma}\widetilde \omega))>0$ and so, $|\widetilde \omega(z)|<1$ for all $z\in E$. This completes the proof.
\end{proof}
As applications of the above theorem we derive the following interesting results.
\begin{corollary} Let $\displaystyle f_{\alpha}=h_{\alpha}+\overline{g}_{\alpha}$, where $h_{\alpha}(z)+g_{\alpha}(z)=\displaystyle {z(1-\alpha z)}/{(1-z^2)},\, \alpha \in [-1,1]$ and $\displaystyle |g_{\alpha}'/h_{\alpha}'|<1$ in $E$, be a normalized harmonic mapping in $E$ and $\phi\in DCP$ be an analytic map. Then $f_\alpha\widetilde \ast \phi \in K_{H(\pi/2)}$.
\end{corollary}
\begin{proof} Let $F_\alpha=h_{\alpha}+g_{\alpha}.$ Since \begin{equation}
\Re[(1-z^2)F'_\alpha(z)]=\Re{\left[\frac{1+z^2-2\alpha z}{(1-z^2)}\right]}=\frac{(1-|z|^2)(1+|z|^2-2\alpha \Re{(z)})}{|1-z^2|^2}>0,\, z\in E,\end{equation} therefore, by Theorem A, the analytic function $F_\alpha$ is univalent in $E$ and convex in the direction of the imaginary axis. Consequently, in view of Lemma 2.3, the harmonic function $f_{\alpha}=h_{\alpha}+\overline{g}_{\alpha}$ is univalent in $E$ and also convex in the direction of the imaginary axis i.e., $f_{\alpha} \in K_{H(\pi/2)}.$ Thus, by keeping (4) in mind, the result immediately follows from Theorem 3.1 by setting  $\eta=-1,$ $\xi=1$ and $\gamma=\pi/2$.
\end{proof}
\begin{corollary} Let $\displaystyle f_{\theta}=h_{\theta}+\overline{g}_{\theta}$, where $h_{\theta}(z)+g_{\theta}(z)=\displaystyle \frac{1}{2i\,\sin\theta} \log\left(\frac{1+ze^{i\theta}}{1+ze^{-i\theta}}\right),$ $\theta\in(0,\pi)$ and $\displaystyle |g_{\theta}'/h_{\theta}'|<1$ in $E$, be a normalized harmonic mapping defined in $E$. Then, for an analytic map $\phi$ in the class DCP, $f_\theta\widetilde\ast \phi \in K_{H(\pi/2)}$.
\end{corollary}
\begin{proof} Let $$\rho(z)=(1-z^2)(h_\theta'+g_\theta')(z)=\frac{1-z^2}{(1+ze^{i\theta})(1+ze^{-i\theta})}.$$ It is easy to verify that $\rho(0)=1$ and for every real number $\delta$, $\Re\rho(e^{i\delta})=0$. Therefore, by minimum principle for harmonic functions, we have, $\Re[\rho(z)]=\Re\left[(1-z^2)(h_\theta'+g_\theta')(z)\right]>0$ in $E.$ Now, the proof follows as in Corollary 3.2.
\end{proof}
 \begin{corollary} Let $\displaystyle f=h+\overline{g}$ be a normalized harmonic mapping such that $h(z)-g(z)=\displaystyle {z}/{(1-z)^2}$ and $\displaystyle |g'/h'|<1$ in $E.$ Then $f\widetilde\ast \phi$ is $CHD$ for every analytic function $\phi\in DCP$.
\end{corollary}
\begin{proof} One can easily verify that $\Re\left[(1-z)^2(h'-g')\right]>0$ in $E$. Further, $h-g=z/(1-z)^2$ is univalent and convex in the horizontal direction and $|g'/h'|<1$ in $E$ implies that $f$ is locally univalent in $E$. Therefore, by Lemma 2.3, $f=h+\overline{g}$ is univalent and CHD in $E$. The desired conclusion now follows immediately from Theorem 3.1 by taking  $\eta=\zeta=1$ and $\gamma=0.$
\end{proof}

\begin{corollary} Let $\displaystyle f=h+\overline{g},$ with $\displaystyle |g'/h'|<1$ in $E,$ be a normalized $(f(0)=0,f_z(0)=1)$ slanted right half-plane mapping in $E$ given by  $$h(z)+e^{-2i\alpha}g(z)=\displaystyle {z}/{(1-e^{i\alpha}z)},\quad -\pi/2\leq\alpha\leq\pi/2,$$ then $f\widetilde\ast \phi \in K_{H(\pi/2-\alpha)}$ for all $\phi\in DCP$.
\end{corollary}
\begin{proof} As $\displaystyle |g'/h'|<1$ in $E,$ so $f$ is locally univalent in $E$. Further for $-\pi/2\leq\alpha\leq\pi/2$
 $$\displaystyle \frac{z}{1-e^{i\alpha}z}=h(z)+e^{-2i\alpha}g(z)=h-e^{2i(\pi/2-\alpha)}g$$ is convex univalent in $E$ and so, in particular, convex in the direction $e^{i(\pi/2-\alpha)}, -\pi/2\leq\alpha\leq\pi/2.$ Hence by Lemma 2.3, $f=h+\overline{g}\in K_{H(\pi/2-\alpha)}, -\pi/2\leq\alpha\leq\pi/2.$ Setting $\eta=\xi=e^{i\alpha}$ and $\gamma=\pi/2-\alpha$ in (2), we get $$\Re[(1-ze^{i\alpha})(1-ze^{i\alpha})(h'-e^{2i(\pi/2-\alpha)}g')]=\Re\left[\frac{(1-ze^{i\alpha})^2}{(1-ze^{i\alpha})^2}\right]>0, z\in E.$$ The result now follows from Theorem 3.1.
\end{proof}
\begin{theorem} Let $f=h+\overline{g}\in K_{H(\pi/2)}$ be such that $F=h + g$ satisfies condition (iii) of Theorem A. Then $f\widetilde\ast \phi \in K_{H(\pi/2)}$ for every $\phi\in DCP.$\end{theorem}
\begin{proof} As $f=h+\overline{g}\in K_{H(\pi/2)},$ therefore, by Lemma 2.3, $F=h+g$ is univalent analytic in $E$ and convex in the direction of the imaginary axis. As $F=h+g$ also satisfies condition (iii) of Theorem A, therefore $\Re[(1-z^2)(h'+g')]>0$ in $E$ and the desired result follows from Theorem 3.1 (taking $\eta=-1,$  $\xi=1$ and $\gamma=\pi/2$).
\end{proof}
 An analytic function $f(z)=z+a_2z^2+...$ is said to be typically-real in $E$ if $f(z)$ is real if and only if $z$ is real in $E$.
 Rogosinski [\ref{ro}] introduced the class $T$ of typically-real functions and proved that a function $f$ is in the class $T$ if and only if $f(z)=z/(1-z^2)P(z)$, where $P$ has real coefficients, $P(0)=1$ and $\Re P(z)>0$ in $E$. Functions in the class $T$ need not be univalent. Let $TK_{\pi/2}$ be the class of univalent functions in $T$ which map the unit disk $E$ onto domains convex in the direction of the imaginary axis.  The following results are known.
 \begin{thmd}   (a) A function $H\in TK_{\pi/2}$ if and only if $zH'\in T$.\\
   (b) If $H$ is in $TK_{\pi/2}$ and $G$ is in $T$, then $H \ast G \in T$.
 \end{thmd}
 \noindent The result (a) in the above theorem is an observation of Fejer [\ref{fe}] and the result (b) is due to Robertson [\ref{rob}].
 We now state and prove our next result.

\begin{theorem} Let $\displaystyle f_{\alpha}=h_{\alpha}+\overline{g}_{\alpha}$, where $h_{\alpha}(z)+g_{\alpha}(z)=\displaystyle {z(1-\alpha z)}/{(1-z^2)},\, \alpha \in [-1,1]$ and $\displaystyle |g_{\alpha}'/h_{\alpha}'|<1$ in $E$, be a normalized harmonic mapping defined in $E$. Then, for every analytic map $\phi_\beta(z)=z(1-\beta z)/(1-z^2), \, \beta\in[-1,1],$ defined in $E$, $f_\alpha\widetilde\ast\phi_\beta \in K_{H(\pi/2)},$ provided $ f_\alpha\widetilde\ast\phi_\beta$ is locally univalent and sense-preserving in $E$.
\end{theorem}
\begin{proof} If $f_\alpha\widetilde\ast\phi_\beta=h_\alpha\ast\phi_\beta+\overline{g_\alpha\ast\phi_\beta}$ is locally univalent and sense-preserving in $E$,  then in view of Lemma 2.3, $f_\alpha\widetilde\ast\phi_\beta \in K_{H(\pi/2)}$ if and only if $h_\alpha\ast\phi_\beta+g_\alpha\ast\phi_\beta$ is univalent and convex in the direction of the imaginary axis. Let $F=h_\alpha\ast\phi_\beta + g_\alpha\ast\phi_\beta=(h_\alpha+g_\alpha)\ast\phi_\beta,$ so that $$
zF'=z(h_\alpha+g_\alpha)'\ast\phi_\beta$$
      Now, if we set $P(z)=1-\beta z,$ then $P(0)=1,$ coefficients of $P$ are real and $\Re P(z)>0$ in $E$ for $\beta\in [-1,1].$ Hence  $$\phi_\beta(z)=\frac{z(1-\beta z)}{1-z^2}\in T.$$ Further, for $\beta\in [-1,1],$ $$\displaystyle \Re[(1-z^2)\phi'_\beta(z)]=\Re{\left[\frac{1+z^2-2\beta z}{1-z^2}\right]}=\frac{(1-|z|^2)(1+|z|^2-2\beta \Re{(z)})}{|1-z^2|^2}>0,\,\, z\in E.$$ Thus, by Theorem A, $\phi_\beta$ is univalent and convex in the direction of the imaginary axis. Hence $\phi_\beta \in TK_{\pi/2}.$ Again $$\displaystyle z(h_\alpha(z)+g_\alpha(z))'=\left[\frac{z}{1-z^2}\frac{(1+z^2-2\alpha z)}{1-z^2}\right]$$ and coefficients of $\displaystyle P_1(z)=\frac{1+z^2-2\alpha z}{1-z^2}$ are real, $P_1(0)=1$ and $\Re P_1(z)=\displaystyle \Re{\left[\frac{1+z^2-2\alpha z}{1-z^2}\right]}>0 $ in $E$. Therefore $z(h_\alpha+g_\alpha)^\prime \in T.$
So in view of Theorem D(b), $zF'\in T$ and hence $F\in TK_{\pi/2}$ by Theorem D (a). This completes the proof.\end{proof}
\begin{remark} We observe the functions $\phi_\beta$ in Theorem 3.7 are not in the class $DCP$ for $\beta \in(-1,1)$ as $\phi_\beta$ are not convex for these values of $\beta$.\end{remark}
\noindent{\emph{Acknowledgement: First author is thankful to the Council of Scientific and Industrial Research, New Delhi, for financial support vide grant no. 09/797/0006/2010 EMR-1.}}
{

\vspace{1cm}
{\footnotesize
{\footnotesize
\noindent{$^a$ Department of Mathematics, DAV University, Jalandhar-144001 (Punjab), India.}}

\noindent{$^b$ Department of Mathematics, Sant Longowal Institute of Engineering and Technology, \\ Longowal-148106 (Punjab), India. }}

\end{document}